\documentclass[12pt,oneside,reqno]{amsart}        

\usepackage{etex}
\reserveinserts{28}  

\usepackage{amssymb,amsthm,amsmath,mathrsfs,comment}
\usepackage{amstext,amsxtra}  
\usepackage{fullpage,colonequals}
\usepackage{bm}       
\usepackage{mathtools} 
\usepackage[all]{xy}
\usepackage{enumitem}

\usepackage{pictex}

\usepackage[pdfpagelabels]{hyperref}
\hypersetup{pdftitle={Classes of order 4 in the strict class group of number fields and
remarks on unramified quadratic extensions of unit type},pdfauthor={David S.\ Dummit}} 
\hypersetup{colorlinks=true,linkcolor=blue,anchorcolor=blue,citecolor=blue}

\theoremstyle{plain}

\newtheorem*{question*}{Question}

\newtheorem{theorem}{Theorem}
\newtheorem{proposition}{Proposition}
\newtheorem{lemma}{Lemma}

\newtheorem*{definition*}{Definition}

\theoremstyle{definition}

\theoremstyle{remark}
\newtheorem{remark}{Remark}
\newtheorem*{remark*}{Remark}

\newcommand{\A}{\mathfrak A}

\newenvironment{enumalph}
{\begin{enumerate}}
{\end{enumerate}}

\def\ZZ{\mathbb Z}
\def\QQ{\mathbb Q}

\def\F{\mathbb F}
\def\OO{\mathcal{O}}   

\def\epsilon{\varepsilon}
\def\phi{\varphi}

\def\f{\mathfrak f}

\def\EKplus{ E_{K}^{+} }
\def\EK4{E_{K,4} }
\def\EKsquares{ E_{K}^{\, \square + \square} }
\def\rank{\text{rank}_2 \, }
\def\norm#1{\text{Norm}_{K/\QQ} ( #1 )} 
\def\order#1{\vert{#1}\vert}    

\DeclareMathOperator{\iso}{\simeq}

\begin{document}

\title[Classes of order 4]
{Classes of order 4 in the strict class group of number fields and
remarks on unramified quadratic extensions of unit type}

\author{David S.\ Dummit}
\address{Department of Mathematics, University of Vermont, Lord House, 16 Colchester Ave., Burlington, VT 05405, USA}
\email{dummit@math.uvm.edu}

\date{\today}

\begin{abstract}
Let $K$ be a number field of degree $n$ over $\QQ$.  Then the 4-rank of the strict class group of 
$K$ is at least $\rank ( \EKplus / E_K^2) -  \lfloor n /2 \rfloor$ where $E_K$
and $\EKplus$ denote the units and the totally positive units of $K$,
respectively, and $\text{rank}_2$ is the dimension as an elementary abelian 2-group.
In particular, the strict class group of a totally real field $K$ 
with a totally positive system of fundamental units contains at least
$(n-1)/2$ ($n$ odd) or $n/2 -1$ ($n$ even) independent elements of order 4.
We also investigate when units in $K$ are sums of two squares in $K$ or are squares mod 4 in $K$.

\bigskip
\noindent
{{\sc R\'esum\'e.}}
Soit $K$ est un corps de nombres de degr\'e $n$ sur $\QQ$.
Alors le 4-rang du groupe de classes au sens strict de $K$ est au moins 
$\text{rang}_2 ( \EKplus / E_K^2) -  \lfloor n /2 \rfloor$
o\`u $E_K$ et $\EKplus$ d\'esignent 
respectivement le groupe des unit\'es et le groupe 
des unit\'es totalement positives de $K$,
et o\`u $\text{rang}_2$ se veut la dimension en tant que 
2-groupe ab\'elien \'el\'ementaire.
En particulier, le groupe de classes au sens strict d'un corps $K$ totalement r\'eel 
avec un syst\`eme fondamental d'unit\'es totalement positives contient au moins
$(n-1)/2$ pour $n$ impair (respectivement $n/2 -1$ pour $n$ pair) \'el\'ements ind\'ependants d'ordre 4.
Nous nous demandons aussi quand les unit\'es de $K$ sont des sommes de deux carr\'es de $K$
ou sont des carr\'es mod 4 de $K$.
\end{abstract}

\subjclass[2010]{11R29 (primary), and 11R37, 11R27, 11E25 (secondary)} 

\maketitle

\section{Introduction}

In 1967 Armitage and Fr\"ohlich proved a result involving the 2-ranks of the usual
class group and the strict (or ``narrow'') class group of a number field $K$.  They showed
in particular that if there are many totally positive units in $K$ then there are independent
elements of order 2 in the class group of $K$.  A result of Hayes in 1997 shows that 
much of this contribution to the class group of $K$ is provided by unramified extensions
$K(\sqrt \epsilon)$ where $\epsilon$ is a totally positive unit in $K$ (the ``unramified quadratic extensions
of unit type'').

In the first part of this paper, we show that the Armitage-Fr\"ohlich theorem implies
that if there are many totally positive units in $K$ then there are independent
elements of order 4 in the strict class group of $K$.

In the second part of this paper, we make some remarks related to the unramified 
quadratic extensions of unit type considered by Hayes.  In particular, the question
of whether they can be embedded in cyclic quartic extensions (which might provide the 
classes of order 4 in the strict class group) leads to consideration of canonical
subgroups of the group of units of a number field.

\section{Classes of order 4 in the strict class group}

Let $K$ be an algebraic number field of degree $n$ over $\QQ$ having $r_1$ real and $r_2$ complex places.
If $C_K$ denotes the class group of $K$ and $C_K^+$ denotes the strict class group of $K$, then there is an
exact sequence
\begin{equation} \label{eq:fundseq}
0 \rightarrow P_K/P_K^+ \rightarrow C_K^+ \rightarrow C_K \rightarrow 0  
\end{equation}
where $P_K$ denotes the group of principal ideals, and $P_K^+$ the group of totally positive principal ideals.
The natural map $\alpha \mapsto (\alpha)$ gives an exact sequence
\begin{equation} \label{eq:fundseqKE}
1 \rightarrow E_K \rightarrow K^* \rightarrow P_K \rightarrow 1 ,
\end{equation}
where $E_K$ denotes the group of units of $K$.  
The image in $P_K$ of the subgroup $K^{*+}$ of totally positive elements of $K^*$ is $P_K^+$, which gives the isomorphism
\begin{equation} \label{eq:principalisom}
P_K/P_K^+ \iso K^* / E_K K^{*+} ,
\end{equation}
so \eqref{eq:fundseq} may be written
\begin{equation} \label{eq:fundseq2}
0 \rightarrow K^* / E_K K^{*+} \rightarrow C_K^+ \rightarrow C_K \rightarrow 0  . 
\end{equation}
The group $K^* / E_K K^{*+}$ is an elementary abelian 2-group, and the sequence \eqref{eq:fundseq2} splits if and only
if the corresponding sequence with $C_K$ and $C_K^+$ replaced by their 2-primary parts splits.

Let $\F_2$ denote the field of order 2 and for a finite abelian group $A$, let $\rank (A) $ denote the 
dimension over $\F_2$ of $A/2A$ (equivalently, the dimension over $\F_2$ of $A[2] = \{ a \in A \mid 2 a = 0 \}$).
By the 4-rank of a finite abelian group $A$ we mean the number of invariant factors of $A$ divisible by
4, or, equivalently, $\rank (2 A)$.

Define
\begin{align*}
\rho & =  \rank C_K \\
\rho^+ & =  \rank C_K^+ \\
\rho_\infty & = \rank ( P_K / P_K^+ )  
\end{align*}

The homomorphism from $K^*$ to $P_K$ in \eqref{eq:fundseqKE} mapping $E_K K^{*+}$ to $P_K^+$, together 
with several natural inclusions, gives the diagram 
$$
\beginpicture
\setcoordinatesystem units <1 pt,1 pt>
\linethickness= 0.3pt
\put {$E_K^2$} at  0 0
\put {$\EKplus$} at  0 40  
\put {$E_K$} at  0 80 
\put {$K^{*+}$} at  50 40      
\put {$E_K K^{*+}$} at  55 80   
\put {$K^{*}$} at  50 120 
\put {$P_K^+$} at  102 80   
\put {$P_K$} at  100 120 
\putrule from 0 10 to 0 30 
\putrule from 0 50 to 0 70 
\putrule from 50 50 to 50 70 
\putrule from 50 90 to 50 110 
\putrule from 98 90 to 98 110 
\arrow<5pt> [.2,.67] from 10 40 to 35 40
\arrow<5pt> [.2,.67] from 10 80 to 35 80
\arrow<5pt> [.2,.67] from 75 80 to 90 80
\arrow<5pt> [.2,.67] from 70 120 to 90 120
\endpicture
$$
from which various rank relations can be determined, as follows.
The signature map taking an element of $K^*$ to its sign in each of the $r_1$ real embeddings of $K$
is a surjective map from $K^*$ to $ \{ \pm 1 \}^{r_1}$ with kernel $K^{*+}$, so $K^{*} / K^{*+}$ is an
elementary abelian 2-group with $\rank (K^{*} / K^{*+} ) = r_1$.  
We have $\rank ( K^* / E_K K^{*+} )  = \rank ( P_K / P_K^+ )  =  \rho_\infty $ by the isomorphism \eqref{eq:principalisom}, and
$ E_K \cap K^{*+} = \EKplus$ shows $E_K K^{*+} / K^{*+} \iso E_K / (E_K \cap K^{*+}) = E_K / \EKplus$
so that $\rank ( E_K K^{*+} / K^{*+} )  = \rank ( E_K / \EKplus )$.  
Since the group 
$E_K / E_K^2$ is an elementary abelian 2-group with $\rank ( E_K / E_K^2 ) = r_1 + r_2$ by the Dirichlet unit theorem, we deduce
the following rank relations:
\begin{align*}
\rank ( K^* / E_K K^{*+} ) & = \rank ( P_K / P_K^+ ) =  \rho_\infty \\
\rank ( E_K K^{*+} / K^{*+} ) & = \rank ( E_K / \EKplus ) = r_1 - \rho_\infty \\
\rank ( \EKplus / E_K^2) & =   r_2 + \rho_\infty . \\
\end{align*}

We now examine more carefully group extensions of abelian 2-groups as in the (2-primary parts of the) 
sequences \eqref{eq:fundseq} and  \eqref{eq:fundseq2}.

\begin{proposition} \label{prop:prop1}
Suppose $B$ is a finite abelian 2-group and $A \le B$ is a subgroup with quotient $C$, so
the sequence $0 \rightarrow A \rightarrow B \rightarrow C \rightarrow 0$ 
is exact, where the first map is the inclusion of $A$ into $B$. Then

\begin{enumalph}

\item
the 4-rank of $B$ is at least $(\rank A + \rank C) - \rank B$,

\item
if $A$ is an elementary abelian 2-group, then 
the maximal rank of a subgroup of $A$ that is a direct summand of $B$ is $\rank B - \rank C$ and
the sequence $0 \rightarrow A \rightarrow B \rightarrow C \rightarrow 0$ splits if and only if
$\rank B = \rank A + \rank C$. 
\end{enumalph}

\end{proposition}

\begin{proof}
Let
$$
A_1 = \left \{ a \in A \mid a = 2 b \text{ for some }b \in B \right \}.
$$
It is easy to see (e.g., by the snake lemma or directly) that the sequence of elementary abelian 2-groups
\begin{equation} \label{eq:2gps}
0 \rightarrow A/A_1 \rightarrow B/2B \rightarrow C/2C \rightarrow 0  
\end{equation}
is exact.  In particular, the sequence splits and
$\rank B = \rank (A/A_1) + \rank C$.  Then 
$\rank A_1 + \rank (A/A_1) \ge \rank A$ gives
$\rank B \ge (\rank A -  \rank A_1) + \rank C$, 
i.e.,
$\rank A_1 \ge (\rank A + \rank C) - \rank B$. 
Since the 4-rank of $B$ is at least $\rank A_1$, this proves (a).   

For (b), note first that
since $B/2B$ is the Frattini quotient for $B$, any collection of elements of $B$ projecting to
a basis for the elementary abelian 2-group $B/2B$ will give a minimal set of generators for $B$. It
follows that a subgroup $A'$ of $A$ is a direct summand of $B$ if and only if $A' \cap A_1 = 0$, 
as follows.  If $B = A' \oplus B'$ for some subgroup $B'$, then
$2 B = 2 B'$, so $A' \cap A_1 = A' \cap 2 B \le A' \cap B' =  0$. 
Conversely, suppose $A' \cap A_1 = 0$.  
Then $A'$ injects into the elementary abelian 2-group $B/2B$ in the exact sequence \eqref{eq:2gps}, hence has a complement.
Let $S$ be a collection of elements of $B$ projecting to a basis for a complement to $A'$, and let $B'$ be the subgroup
of $B$ generated by the elements of $S$.  As noted previously, a basis for $A'$ together with the elements of $S$ give
a minimal set of generators for $B$, so $B = A' + B'$.  Any element in the intersection of $A'$ and $B'$ projects to 0 
in $B/2B$ since $B'$ projects to a complement for $A'$ in
$B/2B$, hence would be an element of $A_1$.   Since $A' \cap A_1 = 0$ by assumption, this shows the sum is direct: 
$B = A' \oplus B'$.

It follows that the maximal rank of a subgroup of $A$ that is a direct summand of $B$ 
occurs for a maximal subgroup of $A$ 
intersecting trivially with $A_1$, i.e., for any complement to $A_1$ in $A$.  The rank of such a complement is
$\rank (A/A_1)$, which is precisely $\rank B - \rank C$ by the exact sequence \eqref{eq:2gps}.

Finally, the sequence splits if and only if the maximal subgroup of $A$ that is a direct summand of
$B$ is $A$, which by the previous result happens if and only if $\rank B - \rank C$ equals $\rank A$. 
\end{proof}

\begin{remark}
It is important that $A$ is an elementary abelian 2-group in (b) of the proposition:
if $B = \ZZ / 8 \ZZ \oplus \ZZ / 2 \ZZ$ and $A$ denotes the 
(non-elementary) cyclic subgroup of order 4 generated by $(2,1) \in B$, then the quotient $C = B/A$ is cyclic of order
4 and both statements of (b) are false.
\end{remark}

Applying Proposition \ref{prop:prop1} to the groups in \eqref{eq:fundseq} and  \eqref{eq:fundseq2} immediately gives

\begin{proposition} \label{prop:4rank}
The sequence $ 0 \rightarrow K^* / E_K K^{*+} \rightarrow C_K^+ \rightarrow C_K \rightarrow 0  $
in  \eqref{eq:fundseq2} splits if and only if $\rho^+ = \rho_\infty + \rho$.  The 4-rank of $C_K^+$ is at least
$(\rho_\infty + \rho ) - \rho^+$.
\end{proposition}

In \cite{A-F}, Armitage and Fr\"ohlich prove that $\rho^+ - \rho \le \lfloor r_1 /2 \rfloor$.  While they explicitly state
only a weaker version of this result (namely that  $\rho_\infty - \rho \le \lfloor r_1 /2 \rfloor$), they prove the
stronger version (see equations (3) and (4) of their paper---note $\rho^+ = \dim_2(X_2)$ and 
$\rho = \dim_2(\text{Ker } \rho \cap X_2)$ in their terminology).  Another proof of the Armitage-Fr\"ohlich 
theorem can be found in Oriat \cite{O} (to whom the stronger version is occasionally incorrectly credited), 
a particularly nice proof due to Hayes (unpublished) can be found in \cite{D-V}, and a
nice proof that highlights the conceptual basis for the result can be found in Greither-Hayes \cite{G-H} (also outlined
in \cite{D-V}).

Using this result we obtain the following theorem.

\begin{theorem} \label{theorem:4rank}
Suppose $K$ is a number field with $\rank ( \EKplus / E_K^2) >  \lfloor n /2 \rfloor$ 
(equivalently,  $\rho_\infty > \lfloor r_1 / 2 \rfloor $). Then

\begin{enumalph}

\item
the sequence $ 0 \rightarrow K^* / E_K K^{*+} \rightarrow C_K^+ \rightarrow C_K \rightarrow 0  $ 
does not split, and 

\item
the 4-rank of the strict class group is at least
$\rank ( \EKplus / E_K^2) -  \lfloor n /2 \rfloor$.

\end{enumalph}
In particular, if $K$ is a totally real number field of degree $n$ with a totally positive system of fundamental
units, then the strict class group of $K$ contains at least
$(n-1)/2$ ($n$ odd) or $n/2 -1$ ($n$ even) independent elements of order 4.

\end{theorem}

\begin{proof}
By the Armitage-Fr\"ohlich theorem, $ \rho^+ - \rho \le \lfloor r_1 / 2 \rfloor $, so
$(\rho_\infty + \rho ) - \rho^+ \ge \rho_\infty - \lfloor r_1 /2 \rfloor$ and (a) and (b) follow from 
Proposition \ref{prop:4rank} since $\rho_\infty - \lfloor r_1 /2 \rfloor = \rank (\EKplus / E_K^2) - \lfloor n/2 \rfloor $.
If $K$ is totally real with a totally positive system of fundamental units then
$\rank (\EKplus / E_K^2) = n-1$, which gives the final statement in the Theorem.  
\end{proof}

\begin{remark}
When $K$ is a real quadratic field, the splitting behavior of \eqref{eq:fundseq} is well-known (cf.\ \cite{Ha})
and follows from the result that
$ \rank (C_K^+) = 1 + \rank (C_K)$ if and only if there is no element $\omega \in K$ with $\norm{\omega} = -1$:
the sequence \eqref{eq:fundseq} splits if and only if either
(a) there is a unit $\epsilon$ with $\norm{\epsilon} = - 1$, in which case $P_K/P_K^+ = 1$, or
(b) there is no element $\omega \in K^*$ with $\norm{\omega} = -1$, in which case $P_K/P_K^+ = \ZZ / 2\ZZ$ and
$C_K^+ \cong C_K \oplus \ZZ / 2\ZZ$; in 
the remaining case (where there is an element $\omega \in K^*$ with $\norm{\omega} = -1$ but
no {\it unit} with the property), then $\order{C_K^+} = 2 \order{C_K}$ but the groups have the same 2-rank.

We briefly recall the proof of the rank result from \cite{Ha}.  Let $I_K$ denote the 
group of nonzero fractional ideals of $K$, let $G$ be the subgroup of ideals $\A$ whose norm agrees, up to sign,
with the norm of some element $\alpha$ of $K^*$: $\norm{\A} = \pm \norm{\alpha}$, and let
$G^+$ be the subgroup of those $\A$ whose norm agrees with the (necessarily positive) norm of
some element $\alpha$ of $K$:  $\norm{\A} =  \norm{\alpha}$.  
If $\tau$ is the nontrivial automorphism of $K$, then $I_K^{1- \tau} \subseteq G^+$
since $\norm{\A/\A^\tau} = 1$. Also, if $\A \in G^+$, then $\norm{\A} =  \norm{\alpha}$ for some 
$\alpha \in K^*$, and since $ \norm{\alpha} > 0$, the principal ideal $(\alpha)$ lies in $P_K^+$. Then
$\norm{ \A/(\alpha) } = 1$, so the prime ideal decomposition of $\A/(\alpha)$ is a product of ideals of the form 
$ {\mathfrak p}/{\mathfrak p}^\tau $ where $ {\mathfrak p} $ is a split prime in $K$.  Hence $\A \in  I_K^{1- \tau} \ P_K^+ $, 
so $G^+ = I_K^{1- \tau} \ P_K^+ $ and it follows that $G^+/P_K^+ = (C_K^+)^2$ since $\tau$ acts by inversion 
on $C_K^+$.  The same argument shows $G/P_K = C_K^2$ (the``Principal Genus Theorem'').
It follows that $I_K/G = (I_K/P_K)/(G/P_K) = C_K / C_K^2$ and similarly 
$I_K/G^+ = C_K^+ /(C_K^+)^2$, so 
$[I_K: G] = 2^{\rank ( C_K )}$ and  $[I_K: G^+] = 2^{\rank ( C_K^+ )}$.
Since $[G:G^+] = 1$ or 2, with $[G:G^+] = 1$ if and only if there is an element
$\omega \in K$ with $\norm{\omega} = -1$, the rank result follows.  

\end{remark}

\begin{remark}
It is perhaps worth noting that the hypotheses of Proposition \ref{prop:prop1}  and Theorem \ref{theorem:4rank}
most likely do not hold for cyclotomic fields, since the signature rank of the units in such
extensions is expected to be nearly maximal.  This is known to be true, for example, for the
cyclotomic fields of prime power conductor $p^n$ for $n$ sufficiently large (\cite{D-D-K}).
\end{remark}

\section{Subfields of the Hilbert class field: unramified extensions of unit type.}

As previously noted, in \cite{A-F} Armitage and Fr\"ohlich explicitly give the inequality 
$\rho \ge \rho_\infty - \lfloor r_1/2 \rfloor $.  
Theorem \ref{theorem:4rank} shows there are $ \rho_\infty - \lfloor r_1/2 \rfloor = \rank (E_K^+ / E_K^2) - \lfloor n/2 \rfloor $
elements of order 4 in the strict class group, and these survive passing to the quotient $C_K$, which `explains' the 
Armitage-Fr\"ohlich result.  In response to a question/conjecture of the author (based on
computations in the case of totally real cubic fields), Hayes (\cite{H}) and then Greither and Hayes (\cite{G-H}) 
proved that a subgroup of 2-rank at least $ \rho_\infty - \lfloor r_1/2 \rfloor $, precisely the contribution to the
class group guaranteed by the explicit theorem of 
Armitage and Fr\"ohlich, is accounted for by totally unramified quadratic 
extensions $K(\sqrt{\epsilon})$ for units $\epsilon \in E_K$ (cf. also \cite{D-V} for a proof of this result).    

The result in Theorem \ref{theorem:4rank}, showing that a large number of totally positive units modulo squares implies the
existence of a number of cyclic quartic extensions of $K$ that are unramified at all 
finite primes, then raises the analogous question of whether one can similarly explicitly 
describe a number of cyclic quartic extensions
from these units.

With an eye to addressing this question, we make some remarks about the ``unramified extensions of unit type'' given
by $K(\sqrt{\epsilon})$ for units $\epsilon \in E_K$.

Let $\OO_K$ denote the ring of integers of $K$, let $\OO_{K,(2)}$ denote its localization at 2, and for every place $v$ of $K$, let 
$\OO_v$ denote the ring of integers in the completion $K_v$ of $K$ at $v$.  

\begin{lemma}  \label{lem:squaremod4}
Suppose $\epsilon$ is a unit in $K$.   Then the following are equivalent:

\begin{enumalph}

\item
There are integers $\alpha, \beta$ in $\OO_K$ with $\epsilon = \alpha^2 + 4 \beta$.

\item
There are 2-integral elements $\alpha_{(2)}, \beta_{(2)}$ in $\OO_{K,(2)}$ with $\epsilon = \alpha_{(2)}^2 + 4 \beta_{(2)}$.
Equivalently, $\epsilon = \alpha_{(2)}^2 (1 + 4 \beta_{(2)})$ for some 2-integral elements $\alpha_{(2)}, \beta_{(2)} \in \OO_{K,(2)}$.

\item
For every place $v$ dividing 2, there are elements $\alpha_v, \beta_v$ in $\OO_v$ with $\epsilon = \alpha_v^2 + 4 \beta_v$.
Equivalently, $\epsilon = \alpha_v^2 (1 + 4 \beta_v)$ for some $\alpha_v, \beta_v$ in $\OO_v$.

\end{enumalph}

\end{lemma}

\begin{proof}
The two statements in (c) are equivalent since $\epsilon = \alpha_v^2 + 4 \beta_v$ implies $\alpha_v$ is locally a unit, so
$\epsilon = \alpha_v^2 (1 + 4 \beta_v/\alpha_v^2)$ and $\beta_v/\alpha_v^2 \in \OO_v$.  Similarly, the two statements in (b) are equivalent. 
Clearly (a) implies (b) implies (c), so it remains to show (c) implies (a).  Suppose for every $v$ dividing 2 there are
elements $\alpha_v, \beta_v$ in $\OO_v$ so that $\epsilon = \alpha_v^2 + 4 \beta_v$.  Choose an integer $\alpha \in \OO_K$ so that for
every $v$ dividing 2, $\alpha \equiv \alpha_v$ mod 4 locally at $v$.  The element $\beta = (\epsilon - \alpha^2)/4$ 
in $K$ is clearly integral at every place of $K$ not dividing 2.  
Since $\beta \in \beta_v + \OO_v  = \OO_v$ for places $v$ dividing 2, $\beta$ is also integral at 2, hence $\beta \in \OO_K$,
so $\epsilon = \alpha^2 + 4 \beta$ with $\alpha, \beta$ in $\OO_K$.
\end{proof}

\begin{definition*}
Let $\EK4$ denote the units in $E_K$ satisfying any of the equivalent conditions in Lemma \ref{lem:squaremod4}, referred to 
as `the units of $K$ that are squares mod 4'.
\end{definition*}

The units $\epsilon$ of $\EK4$ are precisely the units of $K$ such that the extension $K(\sqrt{\epsilon})$ is unramified
over $K$ at all finite primes (see, for example, \cite[Proposition 4.8]{D-V}). 

\begin{remark}
The results in Lemma \ref{lem:squaremod4} show the questions of whether a unit $\epsilon$ is `a square mod 4' in an additive sense
(i.e., $\epsilon = \alpha^2 + 4 \beta$) globally, globally at 2, and locally at primes over 2, are
all equivalent.  By (b) and (c) of the lemma, 
the questions of whether a unit $\epsilon$ is `a square mod 4' in a multiplicative sense
(i.e., $\epsilon = \alpha^2(1 + 4 \beta)$)  globally at 2 and locally at primes over 2 are also
equivalent (and equivalent to the additive version).   
The first example following Theorem \ref{thm:quadfields} below
(see Remark \ref{rem:notglobal}) shows that for $\epsilon \in \EK4$, it may not be possible to 
express $\epsilon$ globally in the form $\alpha^2 (1 + 4 \beta)$ where $\alpha$ and $\beta$ are in $\OO_K$, 
so the conditions in the lemma are not equivalent to this global multiplicative version.  As a result,
some mild care should be exercised with the phrase `$\epsilon$ is a square mod 4'. 
\end{remark}

As motivation for the next definition, recall that a quadratic
extension $K(\sqrt a)$ can be embedded in a cyclic quartic extension of $K$ if and only if $a$ is a sum of two
squares in $K$ (cf.\ for example Exercise 19, Section 14.6 in \cite{D-F}).

\begin{definition*}
Let $\EKsquares$ denote the units in $E_K$ that are the sum of two squares (in $K$):  $\epsilon = \alpha^2 + \beta^2$, 
$\alpha, \beta \in K$. 
\end{definition*}

Each of $\EK4$ and $\EKsquares$ is a canonical subgroup of $E_K$ containing the squares $E_K^2$.

The first part of the following theorem is the result of Hayes mentioned previously.  

\begin{theorem}  \label{thm:unitthm}
With notation as above, we have:

\begin{enumalph}

\item
(Hayes, Greither-Hayes) $\rank (\EK4 \cap \EKplus ) \ge  \rho_\infty - \lfloor r_1/2 \rfloor $. 
Equivalently, there 
are at least $ \rank (E_K^+ / E_K^2) - \lfloor n/2 \rfloor $ independent unramified quadratic extensions
$K(\sqrt{\epsilon})$ of unit type.  

\item
The unit $\epsilon \in \EKplus$ is a sum of two rational squares in $K$, i.e.,  $\epsilon \in \EKsquares$,
if and only if the quadratic Hilbert symbol $(\epsilon, -1)_v$ is 1 at all primes $v$
dividing 2 in $K$ (equivalently, $(\epsilon, \epsilon)_v = 1$, i.e., $\epsilon$ is `isotropic', at all such $v$). 
In particular, every unit in $\EK4 \cap \EKplus$ is a sum of two squares:
$$
E_K^2 \subseteq \EK4 \cap \EKplus \subseteq \EKsquares \subseteq \EKplus \subseteq E_K .
$$

\item
If $g$ denotes the number of primes dividing 2 in $K$, then $\rank (\EKplus / \EKsquares ) \le g - 1 $.  In particular,
for a real quadratic field $K = \QQ( \sqrt {d})$ with squarefree integral $d > 0$ in which 
2 is either ramified or inert, a unit is totally positive if and only if it is a 
sum of two squares in $K$.

\end{enumalph}

\end{theorem}

\begin{proof}
For (a) see \cite{D-V}.

For (b): the unit $\epsilon$ is a sum of two squares, $\epsilon = x^2 + y^2$, if and only if $\epsilon$ is a norm
from the extension $K(i)$ (note that if $i \in K$ then every element $\alpha$, in particular $\epsilon$, 
in $K$ is a sum of two squares: take $x = (\alpha + 1)/2$ and $y = i (\alpha - 1)/2$). 
Hence $\epsilon$ is a sum of two squares if and only if the global 
quadratic Hilbert symbol $(\epsilon, -1)$ is 1, which is the case if and only if each local symbol is 1. 
The symbol $(\epsilon, -1)_v$ is 1 at all infinite places $v$ since $\epsilon \in \EKplus$
is totally positive, and $(\epsilon, -1)_v = 1$ at all odd places $v$ since $K(\sqrt{\epsilon})$ is then
unramified at $v$ and all units are locally norms.  This proves the first statement in (b) (noting that
since $(\epsilon, -\epsilon)$ is always 1, $(\epsilon, -1)_v = 1$ is equivalent to $(\epsilon, \epsilon)_v = 1$). 
For $\epsilon \in \EK4 \cap \EKplus$, the extension $K(\sqrt{\epsilon})$ is also unramified at primes dividing 2,
so just as for odd primes, $(\epsilon, -1)_v = 1$ at all the even places $v$, so every unit in
$\EK4 \cap \EKplus$ is the sum of two squares (in $K$).  

For (c): if $g$ is the number of primes dividing 2, the map 
\begin{align*}
\EKplus & \rightarrow \prod_{v \mid 2} ( \pm 1 ) \ \iso \ ( \pm 1)^g \\
\epsilon & \mapsto (\dots, (\epsilon, -1)_v , \dots ) ,
\end{align*}
defined by the quadratic Hilbert norm residue symbols, is a homomorphism with kernel $\EKsquares$ by (b). The image 
lies in the subgroup of codimension 1 consisting of elements whose product is $+1$ since 
$\prod_{v \mid 2} (\epsilon, -1)_v  = 1$
by the product formula, so $\rank (\EKplus / \EKsquares ) \le g - 1$.
\end{proof}

The ranks of the index relations for these subgroups of the group of units are summarized in the
following diagram. 
$$ 
\underbrace{
\underbrace{
\underbrace{E_K^2 \subseteq \EK4 \cap}_{\ge \rho_\infty- \lfloor r_1/2 \rfloor} 
 \EKplus \subseteq \EKsquares  \underset {\le g - 1} {\subseteq}  }_{= r_2 + \rho_\infty}
 \overbrace{\EKplus \subseteq  E_K  }^{= r_1 - \rho_\infty \ge 1}
}_{= r_1 + r_2}.
$$

\subsection*{Example: Real quadratic fields}

Suppose $K = \QQ( \sqrt {d})$ with squarefree integral $d > 0$.
Then $\rank ( E_K / E_K^2 ) = 2$ and $\rank (E_K / \EKplus) \ge 1$.
If a fundamental unit $\epsilon_K$  has negative
norm (the situation when all possible signatures of units occur), then $\EKplus = E_K^2$, so 
$E_K^2 = \EK4 \cap \EKplus = \EKsquares = \EKplus$ and $\rank (E_K / \EKplus) = 2$, so (b) of Proposition 3 becomes
$$
E_K^2 =  \EK4 \cap \EKplus = \EKsquares 
= \EKplus \underset {4} {\subseteq} E_K ,
$$
where the underscript 4 indicates the index of the subgroup.

Assume for the remainder of the example that the fundamental unit $\epsilon_K$ is totally positive and has been
normalized so that $\epsilon_K > 1$ with respect to the
embedding for which $\sqrt d > 0$.  Then $\norm{\epsilon_K} = +1$ and $\EKplus$ has index 2 in $E_K$:
\begin{equation} \label{eq:unitcontainments}
E_K^2 \subseteq  \EK4 \cap \EKplus \subseteq \EKsquares 
\subseteq \EKplus \underset {2} {\subseteq} E_K .
\end{equation}

By Hilbert's Theorem 90, $\norm{\epsilon_K} = +1$ implies
\begin{equation} \label{eq:alphadef}
\epsilon_K = \dfrac{ \sigma (\alpha) }{\alpha}
\end{equation}
for some $\alpha \in \QQ(\sqrt d)$, which may be assumed to be an algebraic integer (for example, take
$\alpha = \sigma (\epsilon_K) + 1$).   
Since $\sigma (\alpha) = (\alpha)$, the principal ideal $(\alpha)$ is invariant under $\sigma$ (i.e., is
an ambiguous ideal), so the element $\alpha$ in \eqref{eq:alphadef} can be further chosen
so that the ideal $(\alpha)$ is the product of
distinct ramified primes without altering \eqref{eq:alphadef}.  For such a choice of $\alpha$, 
let $m$ denote the norm of $\alpha$:
\begin{equation} \label{eq:mdef}
m = \alpha \ \sigma (\alpha) .
\end{equation}
Then $m$ is a squarefree integer dividing the discriminant of $K = \QQ(\sqrt d)$.

From \eqref{eq:alphadef} we have
\begin{equation} \label{eq:meps}
m \ \epsilon_K = \alpha \ \sigma (\alpha) \dfrac{ \sigma (\alpha) }{\alpha} = \sigma (\alpha)^2
\end{equation}
so 
$m \epsilon_K$ is a square in $K^*$.  We have
\begin{equation} \label{eq:epsplus1}
\epsilon_K + 1 = \dfrac{ \sigma (\alpha) }{\alpha} + 1 = \dfrac{ \alpha + \sigma (\alpha) }{\alpha},
\end{equation}
so
\begin{equation} \label{eq:Normepsplus1}
\norm{\epsilon_K + 1} = \dfrac{ (\alpha + \sigma (\alpha))^2 }{m},
\end{equation}
where $\alpha + \sigma (\alpha) \in \ZZ$.  Hence $m \ \norm{\epsilon_K + 1}$ is a square in $\ZZ$ and it follows that
$m$ is the squarefree part of $\norm{\epsilon_K + 1}$ and that
$m > 1$ since $\epsilon_K + 1 > 1$ with respect to both embeddings of $K$.

We summarize this in the following proposition.

\begin{proposition} \label{prop:m}
Suppose $\norm{\epsilon_K} = +1$ in $\QQ(\sqrt d)$ as above, and let
$m$ denote the squarefree part of the positive integer $\norm{\epsilon_K + 1}$. Then $m > 1$, 
$m$ divides the discriminant of $\QQ(\sqrt d)$, and
$m \epsilon_K$ is a square in $\QQ(\sqrt{d})$.
\end{proposition}

\begin{remark}
Since the discriminant $D$ of $\QQ(\sqrt d)$ is a square in $\QQ(\sqrt d)$, the integers $m$ and $D/m$ differ by a square in
$\QQ(\sqrt{d})$, so the squarefree part of $D/m$ times $\epsilon_K$ is also a square in $\QQ(\sqrt{d})$.  As above,
one can show that the squarefree part of $D/m$ is equal to the squarefree part of the positive integer 
$-\norm{\epsilon_K - 1}$.  Also, if $A$ is $1/2$ of the positive square root of (the positive square integer) $m \norm{\epsilon_K + 1}$ and 
$B$ is $1/2$ of the negative square root of 
(the positive square integer) $- m \norm{\epsilon_K - 1}/d$, then $\alpha$ in \eqref{eq:alphadef} can be taken to be $A + B \sqrt d$.

\end{remark}

The following theorem gives a complete classification of the possible indices in \eqref{eq:unitcontainments}
in terms of the integers $d$ and $m$.

\begin{theorem} \label{thm:quadfields}
Suppose $\epsilon_K$ in $K = \QQ(\sqrt d)$ has norm $+1$ and the positive integer $m$ is as in
Proposition \ref{prop:m}. Then

\begin{enumalph}
\item
$[\EK4 \cap \EKplus : E_K^2] = 2$ (i.e., $\epsilon_K$ is a square mod 4), and so
\begin{equation*}  \label{eq:example3doubleprime}
E_K^2 \underset {2}{\subseteq } \EK4 \cap \EKplus = \EKsquares = \EKplus \underset {2} {\subseteq} E_K ,
\end{equation*}
if and only if one of the following conditions is satisfied:
\begin{enumerate}
\item[{(i)}]
$d \equiv 1$ mod 4 and $m \equiv 1$ mod 4,
\item[{(ii)}]
$d \equiv 3$ mod 4 and $m$ is odd, 
\item[{(iii)}]
$d \equiv 2$ mod 4 and $m \equiv 1$ mod 4, 
\item[{(iv)}]
$d \equiv 2$ mod 8 and $m \equiv 2$ mod 8, or
\item[{(v)}]
$d \equiv 6$ mod 8 and $m \equiv 6$ mod 8,
\end{enumerate}

\smallskip

\item
$[\EKplus : \EKsquares \, ] = 2$ (i.e., $\epsilon_K$ is not the sum of two squares in $K$), and so
\begin{equation*}  \label{eq:example4prime}
 E_K^2 = \EK4 \cap \EKplus = \EKsquares \underset {2} {\subseteq} \EKplus \underset {2} {\subseteq} E_K ,
\end{equation*}
if and only if $d \equiv 1$ mod 8 and $m \equiv 3$ mod 4, and

\smallskip

\item
in all other cases $[\EKsquares : \EK4 \cap \EKplus] = 2$ (i.e., $\epsilon_K$ is the sum of two squares in $K$ but
is not a square mod 4), and so
\begin{equation*}  \label{eq:example3prime}
E_K^2 = \EK4 \cap \EKplus \underset {2} {\subseteq} \EKsquares = \EKplus \underset {2}{\subseteq} E_K .
\end{equation*}
\end{enumalph}

\end{theorem}

\begin{proof}
For (a), we have $[\EK4 \cap \EKplus : E_K^2] = 2$ if and only if $\epsilon_K \in  \EK4 \cap \EKplus $.  
Since $m \epsilon_K$ is a square in $K= \QQ(\sqrt d)$, $K(\sqrt \epsilon_K) = K(\sqrt m) =  \QQ(\sqrt d, \sqrt m)$, 
so $\epsilon_K \in  \EK4 \cap \EKplus $ if and only
if the extension $\QQ(\sqrt d, \sqrt m)/\QQ(\sqrt d)$ is unramified at 2.  

Since $m$ is a squarefree divisor of the discriminant of $K$, there are several possibilities: 
(1) $d \equiv 1$ mod 4 and $m$ is a divisor of $d$,
(2) $d \equiv 3$ mod 4 and $m$ is a divisor of $d$,
(3) $d \equiv 3$ mod 4, and $m = 2 d'$ where $d'$ is a divisor of $d$,
(4) $d \equiv 2$ mod 4  and $m$ is a divisor of $d$,
For (1), the extension $\QQ(\sqrt d, \sqrt m)/\QQ(\sqrt d)$ is unramified at 2 if and only if $m \equiv 1$ mod 4, which is
case (\/{\it i}).
For (2), the extension is always unramified at 2, which is
case (\/{\it ii}), and for (3) the extension is never unramified at 2. 
Finally, for (4), precisely one of $m$ and $d/m$ is odd and the extension is unramified at 2 if and only if
this odd number is $\equiv 1$ mod 4, which gives
cases (\/{\it iii}), (\/{\it iv}) and (\/{\it v}).  
This proves (a).

For (b), $m \epsilon_K$ a square in $K$ implies that $\epsilon_K$ is a sum of two squares in $K$ if and only if
$m$ is a sum of two squares in $K$.  Equivalently, $\epsilon_K$ is a sum of two squares in $K$ if and only if
$m$ is a norm from $K(i) = \QQ( \sqrt d, \sqrt{-1})$ to $K$.  

By the Hasse norm theorem, $m$ is a norm from $K(i)$ to $K$ if and only if $m$ is
locally a norm from $K(i)_v$ to $K_v$ for every place $v$.  Since $m$ is positive, $m$ is locally a norm at the
infinite places.  

If $v$ does not divide 2 and does not divide $m$, then $K(i)_v$ is unramified over $K_v$ and $m$ is a unit at $v$, 
so $m$ is locally a norm at $v$.  

It remains to consider those primes $v$ that divide 2 and those that divide $m$. 
If $\QQ_v = \QQ_p$ for $p$ an odd prime $\equiv 1$ mod 4 dividing $m$, then $\QQ_v(i) = \QQ_p$ and
$K(i)_v = K_v$, so $m$ is locally a norm at $v$. 

For the remaining primes, recall that,
by class field theory, the norms $\alpha_v \in K_v$ from $K(i)_v$ to $K_v$ are the elements of $K_v$ such that
$ \text{Norm}_{K_v/\QQ_v} (\alpha_v)$ is a norm from $\QQ_v(i)$ to $\QQ_v$ (see, for example, \cite[Theorem 7.6]{Iw2}.

If $\QQ_v = \QQ_p$ for $p$ an odd prime $\equiv 3$ mod 4 dividing $m$, then $\QQ_v(i) = \QQ_p(i)$ is the unramified
quadratic extension of $\QQ_p$.  Since $m$ divides $2d$, the prime $p$ divides $d$, so $K_v = \QQ_p(\sqrt d )$ is ramified of degree
2 over $\QQ_p$.  Hence $ \text{Norm}_{K_v/\QQ_v} (m) = m^2$, which is a norm from $\QQ_v(i)$ to $\QQ_v$, so 
$m$ is locally a norm at $v$.

Finally, suppose $v$ is a prime dividing 2.  As before, if $\QQ_2(\sqrt {d}\ )$ is of degree 2 over $\QQ_2$, 
then $ \text{Norm}_{K_v/\QQ_v} (m) = m^2$, which is a norm from $\QQ_2(i)$ to $\QQ_2$, so $m$ is a local norm. 
Hence $m$ is locally a norm at all $v$ unless $\QQ_2(\sqrt {d}\ ) = \QQ_2$, which requires $d \equiv 1$ mod 8.
In this case, $m$ would be odd, and the units that are norms from $\QQ_2(i)$ to $\QQ_2$ are the elements
congruent to 1 mod 4.  Hence $m$ is a local norm from $K(i)_v$ to $K_v$ for every place $v$
unless $d \equiv 1$ mod 8 and $m \equiv 3$ mod 4.  

This completes the proof of (b) and hence also the proof of the proposition since
the three cases are exhaustive and mutually exclusive.  
\end{proof}

The following are explicit examples for each of the three possibilities in Theorem \ref{thm:quadfields}:

\begin{enumerate}

\item[(a)]
$d = 30$, $\epsilon_K = 11 + 2 \sqrt{30}$, $\norm{\epsilon + 1} = 24$, $m = 6$.  Here 
$\epsilon_K$ is a square mod 4:   $ \epsilon_K = (1 +  \sqrt{30})^2 + 4(-5) $
and is a sum of two squares in $K$ (but not the sum of two integral squares):  
$ \epsilon_K = (1 + \sqrt{30}/5)^2 + (2 + 2 \sqrt{30}/5)^2$.
\end{enumerate}

\begin{remark} \label{rem:notglobal}
Note that while $\epsilon_K$ is a square mod 4, it is not globally a square mod 4 in the multiplicative sense:
if $\epsilon_K = \alpha^2 (1 + 4 \beta)$ with integers $\alpha$ and $\beta$ in $K$, then
$\alpha$ would be a unit in $K$, hence of the form $\pm \epsilon_K^n$ for some integer $n$. 
Since $\epsilon_K^2 = 241 + 44 \sqrt{30} \equiv 1$ mod 4, this would imply $\epsilon_K \equiv 1$ mod 4, which
it is not. 

\end{remark}

\begin{enumerate}
\item[(b)]
$d = 33$, $\epsilon_K = 23 + 4 \sqrt{33}$, $\norm{\epsilon + 1} = 48$, $m = 3$. Here
$\epsilon_K$ is not the sum of two squares in $K$.

\item[(c)]
$d = 3$, $\epsilon_K = 2 + \sqrt{3}$, $\norm{\epsilon + 1} = 6$, $m = 6$.  Here
$\epsilon_K$ is the sum of two squares in $K$ (but not the sum of two integral squares): 
$ \epsilon_K = ((1 + \sqrt 3)/2)^2 + (1/2)^2 $, and $\epsilon_K$ is not a square mod 4.

\end{enumerate}

\subsection*{Additional remarks}

The unramified quadratic extensions of unit type have other properties that distinguish them from 
generic quadratic extensions, even among the unramified quadratic extensions of $K$.  For example, they have trivial 
Steinitz class over $K$, a property shared by all the units in $\EK4$ whether or not they are totally positive:

\begin{proposition} \label{prop:Steinitz}
If $\epsilon \in \EK4$, then the ring of integers $\OO_L$ in $L = K(\sqrt \epsilon )$ is free (of rank 2)
as a module over $\OO_K$, in fact $\OO_L = \OO_K + \OO_K (\alpha + \sqrt \epsilon)/2$ where $\epsilon = \alpha^2 + 4 \beta$.
\end{proposition}

\begin{proof}
By Narkiewicz, \cite[Proposition 4.12]{N}, if $A = \OO_K[a] \subseteq \OO_L$ with $a \in \OO_L$ generating $L$ over $K$, then 
$\f_A = \delta_{L/K} (a) \frak D_{L/K}^{-1}$ where
$\f_A = \{ x \in A \mid x \OO_L \subseteq A \}$,  
$\delta_{L/K} (a)$ is the different of the element $a$ for the extension $L/K$ and
$\frak D_{L/K}$ is the different of the extension.
Since $\frak D_{L/K} = \OO_L$ because $L$ is unramified over
$K$, it follows that $A = \OO_L$ if and only if $(\delta_{L/K} (a)) = \OO_L$.  
The element $a = (\alpha + \sqrt \epsilon)/2$ is a root of
the polynomial $x^2 - \alpha x - \beta $, so $\delta_{L/K} (a) = 2 a - \alpha = \sqrt \epsilon$, which generates the trivial 
ideal in $\OO_L$ since $\sqrt \epsilon$ is a unit.  Hence $A = \OO_L$, i.e., $\OO_L = \OO_K[ (\alpha + \sqrt \epsilon)/2 ]$.
\end{proof}

As noted, the result in Proposition \ref{prop:Steinitz} does not require that $\epsilon$ be totally positive (so the
extension $K(\sqrt \epsilon)/K$ may be ramified at infinity).  The extensions
generated by units congruent to squares modulo 4 (but not necessarily totally positive) 
appear in the work of Haggenm\"uller (\cite{Ha1}, \cite{Ha2})
in the context of `free quadratic', or `QF', extensions of $\OO_K$: these are the rings that are free of rank 2 as modules over
$\OO_K$ that are also separable with a Galois action of order 2 in the sense of Galois theory of rings.  
It is noted in Haggenm\"uller (\cite[Lemma 2.2]{Ha1}) that the rings of integers of
quadratic subfields of the Hilbert class field of $K$ include a complete set of representatives for the nonidentity
elements of the group of isomorphism classes of these rings.

It would be interesting to examine the unramified quadratic extensions of unit type with respect to other
arithmetic questions, for example capitulation questions, involving $K$.

\subsection*{Acknowledgments}
I would like to thank Richard Foote and Hershy Kisilevsky for many helpful conversations. I would also like
to thank the anonymous referee, who passed along the observation that the squarefree part of
$\norm{\epsilon +1}$ gives an integer $m$ with $m \epsilon$ a square in $K$ and suggested a complete answer 
for quadratic fields such as in Theorem \ref{thm:quadfields} would be possible (rather than just examples of
each type as in the original manuscript).

\end{document}